\documentclass[final,1p,times]{elsarticle}

\usepackage{amssymb}
 \usepackage{amsthm}
\usepackage{amscd}
\usepackage{amsmath}
\usepackage{amsfonts}
\usepackage{amssymb}
\usepackage{graphicx}
\newtheorem{theorem}{Theorem}
\newtheorem{prop}[theorem]{Proposition}
\newtheorem{lemma}[theorem]{Lemma}
\newtheorem{rmk}[theorem]{Remark}

\usepackage{mathrsfs}
\usepackage{titletoc}

\newcommand{\la}{\Delta}
\def\na{\nabla}
\def\th{\theta}

\renewcommand{\d}{\delta}

\newcommand{\ra}{\rightarrow}
\newcommand{\p}{\partial}
\newcommand{\f}{\frac}
\def\a{\alpha}
\def\lam{\lambda}

\def\e{\epsilon}

\def\v{\varphi}
\def\S{\Sigma}

\def\no{\nonumber\\}
\def\wan{\widetilde}

\newcommand{\be}{\begin{equation}}
\renewcommand{\ra}{\rightarrow}
\newcommand{\ee}{\end{equation}}
\newcommand{\bea}{\begin{eqnarray}}
\newcommand{\eea}{\end{eqnarray}}
\newcommand{\bna}{\begin{eqnarray*}}
\newcommand{\ena}{\end{eqnarray*}}

\renewcommand{\O}{\Omega}
\renewcommand{\le}{\left}
\newcommand{\ri}{\right}

\newcommand\weakto{\rightharpoonup}

\journal{***}

\begin{document}

\begin{frontmatter}

\title{ A generalized Trudinger-Moser inequality on a compact Riemannian surface}

\author{Xiaobao Zhu}
\ead{zhuxiaobao@ruc.edu.cn}
\address{ Department of Mathematics,
Renmin University of China, Beijing 100872, P. R. China}

\begin{abstract}
Let $(\Sigma, g)$ be a compact Riemannian surface. Let $\psi$, $h$ be two smooth functions on $\Sigma$ with
$\int_\Sigma \psi dv_g \neq 0$ and
 $h\geq0$, $h\nequiv0$.
In this paper, using a method of blowup analysis, we prove that the functional
\begin{align}\label{functional_J}
 J^{\psi,h}(u)=\frac{1}{2}\int _{\Sigma}|\nabla_g u|^2dv_g + 8\pi\frac{1}{\int_\Sigma \psi dv_g}\int_\Sigma \psi udv_g-8\pi\log\int _{\Sigma}he^{u}dv_g
\end{align}
is bounded from below in $W^{1,2}(\Sigma,g)$.
Moreover, we obtain a sufficient condition under which $J^{\psi, h}$ attains its infimum in $W^{1,2}(\Sigma, g)$. These results generalize
the main results in \cite{DJLW97} and \cite{YZ2016}.
\end{abstract}


\end{frontmatter}

\titlecontents{section}[0mm]
                       {\vspace{.2\baselineskip}}
                       {\thecontentslabel~\hspace{.5em}}
                        {}
                        {\dotfill\contentspage[{\makebox[0pt][r]{\thecontentspage}}]}
\titlecontents{subsection}[3mm]
                       {\vspace{.2\baselineskip}}
                       {\thecontentslabel~\hspace{.5em}}
                        {}
                       {\dotfill\contentspage[{\makebox[0pt][r]{\thecontentspage}}]}

\setcounter{tocdepth}{2}

\section{Introduction and main results}

Let $(\Sigma, g)$ be a compact Riemannian surface. Let $\psi$, $h$ be two smooth functions on $\Sigma$. In the celebrated paper \cite{DJLW97}, Ding-Jost-Li-Wang studied the functional $J^{\psi,h}$ in $W^{1,2}(\Sigma,g)$ when $\psi\equiv1$ and $h>0$.
Using a method of blowup analysis, they obtained
a sufficient condition ((\ref{condition}) with $\psi\equiv1$)  under which $J^{1,h}$ attains its infimum in $W^{1,2}(\S,g)$.

In this paper, we shall generalize Ding-Jost-Li-Wang's work \cite{DJLW97}. Precisely, we prove the following:
\begin{theorem}\label{thm-m-t}
Let $(\S,g)$ be a compact Riemannian surface. Let $\psi$ be a smooth function on $\Sigma$ satisfying $\int_\Sigma \psi dv_g\neq0$. For any $u\in W^{1,2}(\S, g)$ we have
\begin{align*}
\int_\S e^{u} dv_g\leq C_\S  \exp\left\{\f{1}{16\pi}||\na_g u||_2^2+\wan{u}\right\},
\end{align*}
where $\wan{u}=\frac{1}{\int_\S \psi dv_g}\int_\S \psi u dv_g$ and $C_\S$ is a positive constant depending only on $(\S, g)$.
\end{theorem}
Let $G_y(x)$ be the Green function which satisfies
\begin{align}\label{green-function}
\begin{cases}
\la_g G_y(x)=8\pi\left(\frac{\psi(x)}{\int_\Sigma \psi dv_g}-\d_y(x)\right),~~x\in\S,\\
\int_\S \psi G_y dv_g=0.
\end{cases}
\end{align}
In a normal coordinate system around $y$,  $G_y(x)$ has the expression
\begin{align}\label{green-function-expression}
G_y(x)=&-4\log r+ A_y +b_1 r\cos\theta+b_2 r\sin\theta\nonumber\\
&+c_1r^2\cos^2\theta+2c_2r^2\cos\theta\sin\theta+c_3r^2\sin^2\theta+O(r^3),
\end{align}
where $r(x)=\mbox{dist}(x,y)$ is the distance function from $x$ to $y$ on $(\Sigma, g)$.\\

To prove Theorem \ref{thm-m-t}, we consider the perturbed functional
\begin{align}\label{functional_perturbed}
J^{\psi,h}_\epsilon(u)=\frac{1}{2}\int_\Sigma |\nabla_g u|^2 dv_g + 8\pi(1-\epsilon)\frac{1}{\int_\Sigma \psi dv_g}\int_\Sigma \psi u dv_g -8\pi(1-\epsilon)\log\int_\Sigma he^u dv_g
\end{align}
for $\epsilon\in(0,1)$. In view of the classical Trudinger-Moser inequality (c.f. Lemma \ref{lemma-m-t-ineq} below), $\forall \epsilon\in(0,1)$,
$\exists~ u_\epsilon\in W^{1,2}(\Sigma,g)$, such that $J^{\psi,h}_\epsilon(u_\epsilon)=\inf_{u\in W^{1,2}(\Sigma,g)}J^{\psi,h}_\epsilon(u)$
and $u_\epsilon$ satisfies the Euler-Lagrange equation
\begin{align}\label{E-L_equation}
\Delta_g u_\epsilon = 8\pi(1-\epsilon)\left(\frac{\psi}{\int_\Sigma \psi dv_g} - \frac{he^{u_\epsilon}}{\int_\Sigma he^{u_\epsilon}dv_g}\right).
\end{align}
Without distinguishing sequence and its subsequences, there are two possibilities:\\
$(i).$ If $u_\epsilon$ has a uniform bound in $W^{1,2}(\Sigma,g)$ (i.e., a bound does not depend on $\epsilon$), then $u_\epsilon$ converges to some $u_0$
in $W^{1,2}(\Sigma,g)$ and $u_0$ attains the infimum of $J^{\psi,h}$ in $W^{1,2}(\Sigma,g)$.\\
$(ii).$ If $||u_\epsilon||_{W^{1,2}(\Sigma,g)}\rightarrow\infty$ as $\epsilon\rightarrow0$, one calls $u_\epsilon$ blows up, we shall prove that
\begin{align}\label{bound-from-below}
\inf_{u\in W^{1,2}(\Sigma,g)}J^{\psi,h}(u)\geq-8\pi-8\pi\log\pi-4\pi\max_{y\in \Sigma\setminus Z}\left(2\log h(y)+A_y\right),
\end{align}
where $Z=\{y\in \Sigma:~h(y)=0\}$ and $A_y$ is a smooth function on $\Sigma$ defined in (\ref{green-function-expression}). Combining the results in $(i)$
and $(ii)$ one proves Theorem \ref{thm-m-t}.\\

When $u_\e$ blows up, we construct a blowup sequence $\{\phi_\epsilon\}_{\epsilon>0}$ (c.f. (\ref{test-function})). By a direct calculation one obtains
(\ref{bound-from-above}), letting $\e\ra0$ we have
\begin{align}\label{bound-from-above-1}
\inf_{u\in W^{1,2}(\S,g)}J^{\psi,h}(u) \leq \lim_{\e\ra0}J^{\psi,h}(\phi_\e) = -8\pi-8\pi\log\pi-4\pi\max_{y\in \S\setminus Z}\left(2\log h(y)+A_y\right).
\end{align}
Combining (\ref{bound-from-below}) and (\ref{bound-from-above-1}) we have
\begin{theorem}\label{infimum-of-J}
Let $(\Sigma,g)$ be a compact Riemannian surface. Let $J^{\psi,h}$, $J_\e^{\psi,h}$ and $u_\e$ be defined in (\ref{functional_J}), (\ref{functional_perturbed}) and (\ref{E-L_equation}) respectively.
If $u_\e$ blows up, then we have
\begin{align*}
\inf_{u\in W^{1,2}(\Sigma,g)}J^{\psi,h}(u)=-8\pi-8\pi\log\pi-4\pi\max_{y\in \S\setminus Z}\left(2\log h(y)+A_y\right),
\end{align*}
where $A_y$ is defined in (\ref{green-function-expression}).
\end{theorem}

In view of (\ref{bound-from-above}), if one has (\ref{condition}) below, then we have $J^{\psi,h}(\phi_\e)<\inf_{u\in W^{1,2}(\S,g)}J^{\psi,h}(u)$ for sufficiently small $\e>0$. Then
 Theorem \ref{infimum-of-J} tells us that no blowup happens, so $J^{\psi,h}$ achieves its infimum at some function $u\in W^{1,2}(\S,g)$. Precisely, we obtain the following existence theorem.
\begin{theorem}\label{existence}
Let $(\S,g)$ be a compact Riemannian surface, $K_g$ be its Gaussian curvature.  Let $\psi$, $h$ be two smooth functions on $\Sigma$ with
$\int_\S \psi dv_g \neq0$ and $h\geq0$, $h\nequiv0$. Denote $Z=\{y\in\S:~h(y)=0\}$. Suppose $2\log h(y)+A_y$ attains its supremum in $\S\setminus Z$
at $p$. Let $b_1(p)$ and $b_2(p)$ be the constants in the expression (\ref{green-function-expression}). In a normal coordinate system around
$p$ we write $\na_g h(p)=\left(k_1(p), k_2(p)\right)$. If
\begin{align}\label{condition}
&\la_g h(p)+2\le[b_1(p)k_1(p)+b_2(p)k_2(p)\ri]\nonumber\\
>&-\le[4\pi\le(\f{\psi(p)}{\int_\S \psi dv_g}+1\ri)+\le(b_1^2(p)+b_2^2(p)\ri)-2K_g(p)\ri]h(p),
\end{align}
then the infimum of the functional $J^{\psi,h}$ in $W^{1,2}(\S, g)$ can be attained at some $u\in C^{\infty}(\Sigma)$ which satisfies
\begin{align}\label{equation-u}
\la_g u = 8\pi\left(\frac{\psi}{\int_\S \psi dv_g}-\frac{he^u}{\int_\Sigma he^u dv_g}\right).
\end{align}
\end{theorem}

There are three motivations for the study of this paper:\\
\underline{Motivation 1.}  Ding-Jost-Li-Wang \cite{DJLW97} studied existence of the Kazdan-Warner equation $\la_g u=8\pi-8\pi he^u$
on a compact Riemannian surface with volume $1$. First, they used a variational method to derive a lower bound for $J^{1,h}$ in
$W^{1,2}(\S,g)$; Then, they construct a blowup sequence $\{\phi_\e\}_{\e>0}$ to display that no blowup happens and obtained a sufficient
condition ((\ref{condition}) with $\psi\equiv1$) for the existence of the Kazdan-Warner equation (c.f. \cite{KW74}). Our first motivation is to generalize
these results, see Theorems \ref{thm-m-t} and \ref{existence}. \\
\underline{Motivation 2.} Let $(\S,g)$ be a compact Riemannian surface, $K_g$ be its Gaussian curvature.
The Liouville energy of metric $\wan{g}=e^u g$
 with respect to metric $g$ is represented as $L_g(\wan{g})=\int_\S |\na_g u|^2 dv_g+4\int_\S K_g udv_g$. When $(\S,g)$ is a topological
  two sphere with volume $4\pi$ and bounded curvature $K_g$, Chen-Zhu \cite{CZ13} proved that $L_g(\wan{g})$ is bounded from below in
  $W^{1,2}(\S,g)$.
Their proof  is analytic, does not rely on the uniformization theorem and the Onofri inequality.
In fact, this problem is equivalent to prove that $J^{1,1}$ is bounded from below in $W^{1,2}(\S,g)$.
Our second motivation is  generalize Chen-Zhu's result to general Riemannian surfaces. This is our major motivation.\\
\underline{Motivation 3.} Yang and the author \cite{YZ2016} weakened the condition $h>0$ in \cite{DJLW97}
to $h\geq0$, $h\nequiv0$. Our third motivation is study existence of the generalized Kazdan-Warner equation (\ref{equation-u})
under this condition.\\

We refer the readers to \cite{WL, YangJGA, LWang, WX, BLin} and references therein for more relevant works. \\

\textbf{Concluding remark}: In this paper, we shall follow closely the lines of \cite{DJLW97} and \cite{YZ2016}. We would like to point out two things:
First, in the proof of Theorem \ref{thm-m-t} when we estimate the integral $\int_\S |\na_g u_\e|^2 dv_g$, we divide it into two parts
$$\int_{B_{Rr_\e}(x_\e)}|\na_g u|^2 dv_g~~~~\&~~~~\int_{\S\setminus B_{Rr_\e}(x_\e)}|\na_g u|^2 dv_g$$
instead of three parts
$$\int_{B_{Rr_\e}(x_\e)}|\na_g u|^2 dv_g~~~~\&~~~~\int_{B_{\d}(x_\e)}|\na_g u|^2 dv_g~~~~\&~~~~\int_{\S\setminus B_\d(x_\e)}|\na_g u|^2 dv_g$$ in \cite{DJLW97},
which can simplify the proof of Theorem \ref{thm-m-t}. Second, when $u_\e$ blows up, Ding-Jost-Li-Wang \cite{DJLW97}
proved that $$\inf_{u\in W^{1,2}(\S,g)}J^{1,h}(u)\geq-8\pi-8\pi\log\pi-4\pi\max_{y\in \S}\left(2\log h(y)+A_y\right).$$
we say more about this point. In fact, we shall prove in Theorem \ref{infimum-of-J}
that $$\inf_{u\in W^{1,2}(\S,g)}J^{\psi,h}(u)=-8\pi-8\pi\log\pi-4\pi\max_{y\in \S\setminus Z}\left(2\log h(y)+A_y\right)$$ for a general smooth $\psi$ satisfies
$\int_\S \psi dv_g\neq0$, where $Z=\{ y\in\S:~ h(y)=0\}$. \\

\textbf{\underline{Some main notations:}}\\
\begin{align*}
&\bullet~~ \bar{u}=\frac{1}{\mbox{Vol}_g(\S)}\int_\S u dv_g~~~~~~~~~~~~~~~~~\bullet~~ \wan{u}=\frac{1}{\int_\S \psi dv_g}\int_\S \psi u dv_g\no\\
&\bullet~~ \wan{X}=\le\{u\in W^{1,2}(\S): \wan{u}=0\ri\}~~~~~~~~\bullet~~Z=\{ y\in\S:~ h(y)=0\}\nonumber\\
&\bullet~~ ||\cdot||_p=\left(\int_\S |\cdot|^p dv_g\right)^{1/p},~~L^p-\mbox{norm}~\mbox{on}~(\Sigma,g)\nonumber\\
&\bullet~~ J^{\psi,h}(u)=\frac{1}{2}\int _{\Sigma}|\nabla_g u|^2dv_g + 8\pi\frac{1}{\int_\Sigma \psi dv_g}\int_\Sigma \psi udv_g-8\pi\log\int _{\Sigma}he^{u}dv_g\nonumber\\
&\bullet~~ J^{\psi,h}_\epsilon(u)=\frac{1}{2}\int_\Sigma |\nabla_g u|^2 dv_g + 8\pi(1-\epsilon)\frac{1}{\int_\Sigma \psi dv_g}\int_\Sigma \psi u dv_g -8\pi(1-\epsilon)\log\int_\Sigma he^u dv_g
\end{align*}

The paper is organized as follows: In Section 2, we give three key inequalities. The proof of Theorem \ref{thm-m-t} is given in Section 3.
In Sections 4 and 5, we divide the proof of Theorem \ref{existence} into two parts:
$h>0$ and $h\geq0$, $h\nequiv0$.

Throughout this paper, we use $C$ to denote a positive constant
and its changes from line to line. We do not distinguish sequence and its subsequences in this paper.

\section{ Three key inequalities }

In this section, we present three key inequalities which are very important in the following study.

\subsection{ The Trudinger-Moser inequality on a compact Riemannian surface }

\begin{lemma}\label{lemma-m-t-ineq}(\cite{F93, DJLW97})
Let $(\S,g)$ be a compact Riemannian surface. For any $u\in W^{1,2}(\S, g)$ we have
\begin{align}\label{m-t-ineq}
\int_\S e^{u} dv_g\leq C_\S  \exp\left\{\f{1}{16\pi}||\na_g u||_2^2+\bar{u}\right\},
\end{align}
where $\bar{u}=\f{1}{\mbox{Vol}_g(\S)}\int_\S u dv_g$ and $C_\S$ is a positive constant depending only on $(\S, g)$.
\end{lemma}

For improvements of the above Trudinger-Moser inequality, we refer the readers to
Adimurthi-Druet \cite{A-D}, Yang \cite{YangTran,YangJDE},
Lu-Yang \cite{Lu-Yang}, Wang-Ye \cite{WY}, Yang-Zhu \cite{YZ-AGAG}, Tintarev \cite{Tint} and the author \cite{Zhu}.

\subsection{ The  Poincar\'e type inequality }
\begin{lemma}\label{lemma-poincare-ineq}
Let $(\S,g)$ be a compact Riemannian surface. Let $\psi$ be a smooth function on $\Sigma$ satisfying $\int_\Sigma \psi dv_g\neq0$.
Assume $q>1$, then for any $u\in W^{1,q}(\S, g)$ we have
\begin{align}\label{poincare-ineq}
\le(\int_\S |u-\wan{u}|^q dv_g\ri)^{1/q}\leq C_\S\le(\int_\S |\na_g u|^q dv_g\ri)^{1/q}
\end{align}
 where $\wan{u}=\f{1}{\int_\S \psi dv_g}\int_\S \psi u dv_g$ and $C_\S$ is a positive constant depending only on $(\S, g)$.
\end{lemma}
Since the proof of Lemma \ref{lemma-poincare-ineq} is completely analogous to the case that $\psi$ is a constant, we omit it here
and refer the readers to Theorem 2.10 in \cite{Hebey}.

\subsection{ The Sobolev-Poincar\'e type inequality }

\begin{lemma}\label{lemma-s-p}
Let $(\S,g)$ be a compact Riemannian surface. Let $\psi$ be a smooth function on $\S$ with  $\int_\S \psi dv_g\neq0$.
Assume $p\geq1$, then for any $u\in W^{1,2}(\S, g)$ we have
\begin{align*}
\le(\int_\S |u-\wan{u}|^p dv_g\ri)^{1/p}\leq C_\S\le(\int_\S |\na_g u|^2 dv_g\ri)^{1/2}.
\end{align*}
 where $\wan{u}=\f{1}{\int_\S \psi dv_g}\int_\S \psi u dv_g$ and $C_\S$ is a positive constant depending only on $(\S, g)$.
\end{lemma}
\begin{proof}
The proof is standard. Suppose not, there exists a sequence of functions $\{u_n\}_{n=1}^{\infty}\subset W^{1,2}(\S, g)$ such that
\begin{align*}
\le(\int_\S |u_n-\wan{u}_n|^p dv_g\ri)^{1/p}\geq n\le(\int_\S |\na_g u_n|^2 dv_g\ri)^{1/2}.
\end{align*}
Let
$$v_n=\f{u_n-\wan{u}_n}{\le(\int_\S |u_n-\wan{u}_n|^p dv_g\ri)^{1/p}}.$$
Easily check can find that
\begin{align*}
||v_n||_p=1,~~||\na_g v_n||_2\leq\f{1}{n},~~\wan{v}_n=0.
\end{align*}
By Lemma \ref{lemma-poincare-ineq} we have
$$||v_n||_2\leq C.$$
So
$$||v_n||_{W^{1,2}(\S, g)}\leq C.$$
Therefore,
\begin{align}\label{s-p-1}
&v_n\weakto v_0~~\mbox{weakly~in}~W^{1,2}(\S, g),\nonumber\\
&v_n\to v_0~~\mbox{strongly~in}~L^q(\S, g)~(\forall q\geq1).
\end{align}
By the lower semi-continuous property of $||\na_g\cdot||_2^2$ one knows $||\na_g v_0||_2^2=0$, so $v_0$ is a constant.
From (\ref{s-p-1}), we have
\begin{align}\label{s-p-2}
\wan{v}_0=0
\end{align}
and
\begin{align}\label{s-p-3}
||v_0||_p=1.
\end{align}
Since $v_0$ is a constant, (\ref{s-p-2}) tells us that $v_0\equiv0$. This is a contraction with (\ref{s-p-3}).
This ends the proof of the lemma.
\end{proof}

\section{ Proof of Theorem \ref{thm-m-t} }

In this section, we shall derive the lower bound of $J^{\psi,h}$. As a consequence, we shall prove Theorem \ref{thm-m-t}.\\

Since $J_\e^{\psi,h}(u+c)=J_\e^{\psi,h}(u)$ and $J^{\psi,h}(u+c)=J^{\psi,h}(u)$ for any $\e\in(0,1)$, $c\in\mathbb{R}$ and any $u\in W^{1,2}(\S, g)$, we have
$$\inf_{u\in\wan{X}}J_\e^{\psi,h}(u)=\inf_{u\in W^{1,2}(\S, g)}J^{\psi,h}_\e(u),~~\inf_{u\in\wan{X}}J^{\psi,h}(u)=\inf_{u\in W^{1,2}(\S, g)}J^{\psi,h}(u),$$
where $\wan{X}=\le\{u\in W^{1,2}(\S): \wan{u}=0\ri\}$. Therefore, we can without loss of generality assume that $u_\e\in \wan{X}$. There are two possibilities:\\

\textbf{Case a).} $||\na_g u_\e||_2\leq C$.\\

Since $u_\e\in \wan{X}$, by the Poincar\'e inequality (\ref{poincare-ineq}) we know
$u_\e$ is bounded in $W^{1,2}(\S,g)$. Then we may assume
\begin{align}\label{w-e-conv}
&u_\e\rightharpoonup u_0~~\mbox{weakly~in}~~W^{1,2}(\S, g),\nonumber\\
&u_\e\ra u_0~~\mbox{strongly~in}~~L^p(\S, g),\,\,\forall p\geq1.
\end{align}
This together with the Trudinger-Moser inequality (\ref{m-t-ineq}) and the H\"{o}lder inequality
leads to
\begin{align}
\label{h-conv}
\int_\S h\le(e^{u_\e}-e^{u_0}\ri)dv_g
&=\int_\S h \int_0^1 \frac{d}{dt}e^{u_0+t(u_\e-u_0)}dtdv_g\nonumber\\
&=\int_0^1 \int_\S h e^{u_0+t(u_\e-u_0)}(u_\e-u_0)dv_g dt\nonumber\\
&\ra0~~as~~\e\ra0.
\end{align}
From (\ref{w-e-conv}) and the H\"{o}lder inequality we have
\begin{align}\label{K_g-conv}
\int_\S \psi (u_\e-u_0)dv_g\ra 0~~as~~\e\ra0.
\end{align}
The lower semi-continuous of $||\na_g\cdot||^2_2$ together with (\ref{w-e-conv})-(\ref{K_g-conv}) gives us
\begin{align*}
\inf_{u\in \wan{X}}J^{\psi,h}(u)\geq\liminf_{\e\ra0}\inf_{u\in\wan{X}}J^{\psi,h}_\e(u)=\liminf_{\e\ra0}J^{\psi,h}_\e(u_\e)\geq J^{\psi,h}(u_0).
\end{align*}
That is to say, $u_0\in \wan{X}$ attains the infimum of $J^{\psi,h}$ in $\wan{X}$ and satisfies (\ref{equation-u}).
The elliptic regularity theory implies that $u_0\in C^{\infty}(\S)$. The proof of Theorems \ref{thm-m-t} and \ref{existence} terminates in this case.\\

\textbf{Case b).} $||\na_g u_\e||_2\ra+\infty$ as $\e\ra0$.\\

Though the $L^2$-norm of the gradient of $u_\e$ is infinity, we have
\begin{lemma}\label{lemma-l_q-norm}
For any $1<q<2$, $||\na_g u_\e||_q\leq C.$
\end{lemma}
\begin{proof}
Let $q'=q/(q-1)>2$, by equation (\ref{E-L_equation}) we have
\begin{align*}
||\na_g u_\e||_q
&\leq \sup_{||\zeta||_{W^{1, q'}(\S)}\leq1,~\int_\S \zeta dv_g=0} \int_\S \na_g u_\e \na_g \zeta dv_g\nonumber\\
&=\sup_{||\zeta||_{W^{1, q'}(\S)}\leq1,~\int_\S \zeta dv_g=0} \int_\S 8\pi(1-\epsilon)\left(-\frac{\psi}{\int_\Sigma \psi dv_g} + \frac{he^{u_\epsilon}}{\int_\Sigma he^{u_\epsilon}dv_g}\right) \zeta dv_g\nonumber\\
&\leq C,
\end{align*}
where in the last inequality we have used the Sobolev embedding $W^{1,q'}(\S, g)\hookrightarrow C^0(\S)$.
\end{proof}

Denote $\lam_\e=\int_\S he^{u_\e}dv_g$. We have
\begin{lemma}\label{l>0}
$\liminf_{\epsilon\ra 0}\lambda_\epsilon>0$.
\end{lemma}
\begin{proof}
\begin{align}\label{lll}
J_\e^{\psi,h}(u_\epsilon)=\inf_{u\in \wan{X}}J_\e^{\psi,h}(u)\leq
J_\e^{\psi,h}(0)\leq -8\pi(1-\e)\log\int_\Sigma hdv_g.
\end{align}
If $\liminf_{\epsilon\ra 0}\lambda_\epsilon=0$, then up to a subsequence we have
$$J_\e^{\psi,h}(u_\epsilon)=\f{1}{2}\int_\S |\na_g u_\epsilon|^2 dv_g - 8\pi(1-\epsilon)\log\lambda_\epsilon\ra +\infty$$
as $\e\ra0$, which contradicts (\ref{lll}). This ends the proof of Lemma \ref{l>0}.
\end{proof}
Let  $c_\e=\max_\S u_\e = u_\e(x_\e)$. Suppose $x_\e\ra p$ as $\e\ra0$, then
\begin{lemma}\label{lemma-c}
$c_\epsilon\ra+\infty$ as $\e\ra0$. Furthermore, we have $\lam_\e^{-1}e^{c_\e}\ra+\infty$ as $\e\ra0$.
\end{lemma}
\begin{proof}
 Multiplying both sides of the equation (\ref{E-L_equation}) by $u_\epsilon$ and integrating both sides on $(\S,g)$, we have
\begin{align}\label{lemma-c_1}
\int_\S |\na_g u_\e|^2 dv_g \leq 8\pi(1-\e)c_\e.
\end{align}
This implies $c_\epsilon\ra+\infty$ as $\e\ra0$ since $||\na_g u_\e||_2\ra+\infty$ as $\e\ra0$.

By Lemmas \ref{lemma-m-t-ineq} and \ref{lemma-poincare-ineq} one has
\begin{align}\label{lemma-c_2}
\lam_\e=\int_\S he^{u_\e} dv_g
\leq& C \exp\left\{\f{1}{16\pi}||\na_g u_\e||^2_2+\bar{u}_\e\right\}\nonumber\\
\leq& C \exp\left\{\left(\f{1}{16\pi}+\d\right)||\na_g u_\e||_2^2+C_\d\right\}.
\end{align}
Substituting (\ref{lemma-c_1}) into (\ref{lemma-c_2}) and choosing $\d=\f{1+2\e}{32\pi(1-\e)}$, then
\begin{align*}
\lam_\e^{-1} e^{c_\e}\geq C e^{\f{1}{4}c_\e}\ra+\infty~~\mbox{as}~\e\ra0.
\end{align*}
This ends the proof of Lemma \ref{lemma-c}.
\end{proof}

Choosing a local coordinate system $(U,z)$ around $p$, which satisfies $z(p)=0$.
Let $r_\e=\sqrt{\lam_\e}e^{-c_\e/2}$ and define
\begin{align*}
\varphi_\e(x)=u_\e(z(x_\e)+r_\e x)-\lam_\e,~~~~x\in \mathbb{B}_R(0)\subset z(U).
\end{align*}
Then in $\mathbb{B}_R(0)$,
\begin{align*}
\la_{g}\varphi_\e(x)=&8\pi(1-\e)\left(\f{\psi(z(x_\e)+r_\e x)}{\int_\S \psi dv_g}r_\e^2-h(z(x_\e)+r_\e x)e^{\varphi_\e(x)}\right)\nonumber\\
:=&f_\e(x).
\end{align*}
We have the following asymptotic phenomenon of $u_\e$ near the blowup point $p$.
\begin{lemma}\label{bubble}
\begin{align}\label{bubble-function}
\v_\e(x)\ra\v_0(x)=-2\log(1+\pi h(p)|x|^2)
\end{align}
in $C^{1}_{loc}(\mathbb{R}^2)$ as $\e\ra0$.
\end{lemma}
\begin{proof}
Since $f_\e\in L^{\infty}(\mathbb{B}_R(0))$, by Theorem 9.15 in \cite{GT} we can consider the equation
\begin{align*}
\begin{cases}
      & \la_{g}\varphi_{\e}^{1}(x)=f_\e(x),~~ x\in \mathbb{B}_{R}(0), \\
      & \varphi_{\e}^{1}|_{\partial \mathbb{B}_{R}(0)}=0.
\end{cases}
\end{align*}
Let $\varphi_{\e}^{2}=\varphi_{\e}-\varphi_{\e}^{1}$, then $\la_{g}\varphi_{\e}^{2}=0$. Since $\v_\e\leq0$, $f_\e$ is bounded in $\mathbb{B}_R(0)$.
 The elliptic estimates together
with $W_0^{2,p}(\mathbb{B}_{R}(0))\hookrightarrow C(\overline{\mathbb{B}_{R}(0)})$ give $\sup_{\mathbb{B}_{R}(0)}|\varphi_{\e}^{1}|\leq C$. Because
$\varphi_{\e}\leqslant0$, we have $\sup_{\mathbb{B}_{R}(0)}\varphi_{\e}^{2}\leqslant C$. The Harnack inequality yields that
$\sup_{\mathbb{B}_{\frac{R}{2}}(0)}|\varphi_{\e}^{2}|\leqslant C$, because $\varphi_{\e}^{2}(0)$ is bounded. Therefore $\sup_{\mathbb{B}_{\frac{R}{2}}(0)}|\varphi_{\e}|\leqslant C$.

By the elliptic estimates, we can show that $\varphi_{\e}(x)\ra\varphi_{0}(x) \, \, in \,\,C^{1}_{loc}(\mathbb{R}^2)$ as $\e\ra0$, where $\varphi_{0}(x)$ satisfies
\begin{align*}
\begin{cases}
      & \la_{\mathbb{R}^2}\varphi_{0}=-8\pi h(p)e^{\varphi_{0}}, \\
      & \varphi_{0}(0)=0, \\
      & \int_{\mathbb{R}^{2}}h(p)e^{\varphi_{0}}dx\leq1.
\end{cases}
\end{align*}
By Chen-Li's classification theorem \cite{CL} we know
\begin{align*}
\v_{0}(x)=-2\log\left(1+\pi h(p)|x|^2\right).
\end{align*}
This is the end of the proof of Lemma \ref{bubble}.
\end{proof}

Away from the blowup point $p$, we have
\begin{lemma}\label{lemma-u-to-g}
For any $\O\subset\subset\S\setminus\{p\}$, we have $||u_\e||_{L^{\infty}(\O)}\leq C$.
\end{lemma}
\begin{proof}
 Let $\O\subset\subset \S\setminus\{p\}$. We choose another two compact sets $\O_1$
and $\O_2$ in $\S\setminus\{p\}$ such that $\O\subset\subset\O_1\subset\subset\O_2\subset\subset\S\setminus\{p\}$.
Calculating directly, one knows in Lemma \ref{bubble}
\begin{align}\label{in-is-1}
\lim_{R\ra+\infty}\lim_{\e\ra0}\lam_\e^{-1}\int_{B_{Rr_\e}(x_\e)} he^{u_\e} dv_g = \int_{\mathbb{R}^2} h(p)e^{\varphi_0}dx=1.
\end{align}
Then by (\ref{in-is-1}) one has
\begin{align}\label{out-is-0}
\lim_{R\ra+\infty}\lim_{\e\ra0}\lam_\e^{-1}\int_{\S\setminus B_{Rr_\e}(x_\e)} he^{u_\e} dv_g = 0.
\end{align}
So we have
\begin{equation}
\label{O_2}
\lim_{\e\ra0}\lam_\e^{-1}\int_{\O_2} he^{u_\e}dv_g=0.
\end{equation}
Assume $u_\e^1$ be the unique solution of
\begin{equation}
\label{u_e-1}
\begin{cases}
      & \la_g u_\e^1=-8\pi(1-\e)\lam_\e^{-1}he^{u_\e}  \text{~~in~~$\O_2$}, \\
      & u_\e^1=0  \text{~~on~~$\p \O_2$}.
\end{cases}
\end{equation}
From (\ref{O_2}) and Theorem 1 in \cite{BM91} one knows that for some $q\in(1,2)$, we have
\begin{equation}
\label{e-u-1}
\int_{\O_2}e^{q|u_\e^1|} dv_g\leq C.
\end{equation}
It follows that in (\ref{u_e-1}) one has
$$||u_\e^1||_{L^q(\O_2)}\leq C.$$

Let $u_\e^2=u_\e-u_\e^1$. Then $\la_g u_\e^2=8\pi(1-\e) \psi /\int_\S \psi dv_g$ in $\O_2$. It follows from the interior $L^p$-estimates
 (c.f. Theorem 8.17 in \cite{GT}) and Lemma \ref{lemma-l_q-norm}  that
\begin{align}
\label{u_e-2}
||u_\e^2||_{L^{\infty}(\O_1)}\leq& C||u_\e^2||_{L^q(\O_2)}\nonumber\\
\leq& C\le(||u_\e||_{L^q(\O_2)}+||u_\e^1||_{L^q(\O_2)}\ri)\nonumber\\
\leq& C\le(||\na_g u_\e||_{L^q(\S)}+||u_\e^1||_{L^q(\O_2)}\ri)\nonumber\\
\leq& C.
\end{align}
By combining (\ref{e-u-1}) and (\ref{u_e-2}) we have
\begin{align*}
\label{}
\int_{\O_1}e^{pu_\e}dv_g=\int_{\O_1} e^{pu_\e^1} e^{pu_\e^2} dv_g\leq C.
\end{align*}
Using the standard elliptic estimates to equation (\ref{u_e-1}), one obtains
\begin{align*}
\label{}
||u_\e^1||_{L^{\infty}(\O)}\leq C.
\end{align*}
Thus,
\begin{equation}\label{out-infty-norm}
||u_\e||_{L^{\infty}(\O)}\leq C.
\end{equation}
This ends the proof of Lemma \ref{lemma-u-to-g}.
\end{proof}
\begin{rmk}
In fact, one can show that $u_\e\ra G_p$ in $C^{1}_{loc}(\S\setminus \{p\})$ as $\e\ra0$, like
in \cite{DJLW97}. However, we do not need this fact, so we omit it here.
\end{rmk}

To complete the proof of Theorem \ref{thm-m-t}, we still need a lower bound for $u_\e$ away from the maximum point
$x_\e$.

 Similar to \cite{DJLW97}, we have the following lemma by the maximum principle.
\begin{lemma}\label{m-p}
For any fixed $R>0$, let $r_\e=\sqrt{\lam_\e}e^{-c_\e/2}$. Then for any $y\in \S\setminus B_{Rr_\e}(x_\e)$, we have
$$u_\e(y)-G_{x_\e}(y)\geq -c_\e+2\log\lam_\e
-2\log\pi-2\log h(p)-A_p+o_\e(1)+o_R(1),$$
where $o_\e(1)\ra0$ as $\e\ra0$ and $o_R(1)\ra0$ as $R\ra+\infty$.
\end{lemma}
\begin{proof}
By (\ref{E-L_equation}) and (\ref{green-function}), we have
 for any $y\in \S\setminus B_{Rr_\e}(x_\e)$
$$\la_g\le(u_\e-G_{x_\e}\ri)=-8\pi\e\f{\psi}{\int_\S \psi dv_g}-8\pi(1-\e)\lam_\e^{-1}he^{u_\e}.$$
Notice that one needs to deal with the term $-8\pi\e\psi/\int_\S \psi dv_g$ which sign is unknown. Employing the trick introduced by Chen-Zhu (c.f. Lemma 3.5 in \cite{CZ13}), we define
\begin{align*}
\psi'=
\begin{cases}
\f{\psi}{\int_\S \psi dv_g}, ~~~~\text{if}~~\f{\psi}{\int_\S \psi dv_g}\leq0,\\
\phi\f{\psi}{\int_\S \psi dv_g},~~\text{if}~~\f{\psi}{\int_\S \psi dv_g}>0,
\end{cases}
\end{align*}
where $\phi(x)\in[0,1]$ is a measurable function on $\S$ such that $\int_\S \psi'dv_g=0$. Assume $\la_g \zeta=\psi'$ on $\S$. Then $\zeta$
is bounded on $\S$. Consider the function $u_\e-G_{x_\e}+8\pi\e\zeta$, we have
$$\la_g\le(u_\e-G_{x_\e}+8\pi\e\zeta\ri)=-8\pi\e\left(\f{\psi}{\int_\S \psi dv_g}-\psi'\right)-8\pi(1-\e)\lam_\e^{-1}he^{u_\e}\leq0,~~y\in\S\setminus B_{Rr_\e}(x_\e).$$
Lemma \ref{bubble} together with (\ref{green-function-expression}) tells us that
$$\le(u_\e-G_{x_\e}+8\pi\e\zeta\ri)\le|_{\p B_{Rr_\e}(x_\e)}\ri.=-c_\epsilon+2\log\lambda_\epsilon
-2\log\pi-2\log h(p)-A_{x_\epsilon}+o_\epsilon(1)+o_R(1).$$
Then by the maximum principle we know
\begin{align*}
u_\e-G_{x_\e}+8\pi\e\zeta\geq -c_\epsilon+2\log\lambda_\epsilon
-2\log\pi-2\log h(p)-A_{x_\epsilon}+o_\epsilon(1)+o_R(1).
\end{align*}
Since $\zeta$ is bounded and $A_{x_\e}\ra A_p$ as $\e\ra0$, we have
\begin{align*}
u_\e-G_{x_\e}\geq -c_\epsilon+2\log\lambda_\epsilon
-2\log\pi-2\log h(p)-A_{p}+o_\epsilon(1)+o_R(1).
\end{align*}
This ends the proof of the lemma.
\end{proof}

We are now ready to estimate $J^{\psi,h}_\e(u_\e)$ from below and give the proof of Theorem \ref{thm-m-t}.\\

\underline{\textit{Proof of Theorem \ref{thm-m-t}.}} Recall that $r_\e=\sqrt{\lam_\e}e^{-c_\e/2}$, for any fixed $R>0$, we have
\begin{align}\label{energy}
\int_\S |\na_g u_\e|^2 dv_g=\int_{\S\setminus B_{Rr_\e}(x_\e)} |\na_g u_\e|^2dv_g+\int_{B_{Rr_\e}(x_\e)}|\na_g u_\e|^2dv_g.
\end{align}
From Lemma \ref{bubble} one knows
\begin{align}\label{energy-in}
\int_{B_{Rr_\e}(x_\e)}|\na_g u_\e|^2dv_g
=&\int_{B_R(0)}|\na_{\mathbb{R}^2} \varphi_0|^2dx+o_\e(1)\nonumber\\
=&16\pi\log\le(1+\pi h(p)R^2\ri)-16\pi+o_\e(1)+o_R(1).
\end{align}
By the equation of $u_\e$ (c.f. (\ref{E-L_equation})) we have
\begin{align}
\label{out}
\int_{\S\setminus{B_{Rr_\e}(x_\e)}}|\na_g u_\e|^2dv_g
=&-\int_{\S\setminus{B_{Rr_\e}(x_\e)}}u_\e \la_g u_\e dv_g-\int_{\p B_{Rr_\e}(x_\e)}u_\e\f{\p u_\e}{\p n}ds_g\nonumber\\
=&8\pi(1-\e)\lam_\e^{-1}\int_{\S\setminus{B_{Rr_\e}(x_\e)}} u_\e he^{u_\e} dv_g\nonumber\\
  &-\frac{8\pi(1-\e)}{\int_\S \psi dv_g}\int_{\S\setminus{B_{Rr_\e}(x_\e)}}\psi u_\e dv_g-\int_{\p B_{Rr_\e}(x_\e)}u_\e\f{\p u_\e}{\p n}ds_g.
\end{align}
Lemma \ref{m-p} leads to
\begin{align}
\label{out0}
&8\pi(1-\e)\lam_\e^{-1}\int_{\S\setminus{B_{Rr_\e}(x_\e)}} u_\e he^{u_\e} dv_g\nonumber\\
\geq &8\pi(1-\e)\lam_\e^{-1}\int_{\S\setminus{B_{Rr_\e}(x_\e)}}G_{x_\e} he^{u_\e}dv_g-8\pi(1-\e)\lam_\e^{-1}\left(c_\e-2\log\lam_\e\right)\int_{\S\setminus{B_{Rr_\e}(x_\e)}}  he^{u_\e}dv_g\nonumber\\
&+ 8\pi(1-\e)\lam_\e^{-1}\int_{\S\setminus{B_{Rr_\e}(x_\e)}}\le(-2\log\pi-2\log h(p)-A_{p}+o_\epsilon(1)+o_R(1)\ri) he^{u_\e} dv_g.
\end{align}
From (\ref{E-L_equation}) and (\ref{green-function}) one has
\begin{align}
\label{out1}
&8\pi(1-\e)\lam_\e^{-1}\int_{\S\setminus{B_{Rr_\e}(x_\e)}}G_{x_\e} he^{u_\e} dv_g\nonumber\\
 =&\int_{\S\setminus{B_{Rr_\e}(x_\e)}}G_{x_\e}\le(-\la_gu_\e + 8\pi(1-\e)\frac{\psi}{\int_\S \psi dv_g}\ri)dv_g\nonumber\\
 =&\frac{8\pi}{\int_\S \psi dv_g}\int_{B_{Rr_\e}(x_\e)}u_\e\psi dv_g- \frac{8\pi(1-\e)}{\int_\S \psi dv_g} \int_{B_{Rr_\e}(x_\e)}G_{x_\e}\psi dv_g\nonumber\\
 &+\int_{\p B_{Rr_\e}(x_\e)}G_{x_\e}\f{\p u_\e}{\p n}ds_g-\int_{\p B_{Rr_\e}(x_\e)}u_\e \f{\p G_{x_\e}}{\p n}ds_g.
\end{align}
It also follows from  (\ref{E-L_equation}) that
\begin{align}
\label{out2}
&-8\pi(1-\e)\lam_\e^{-1}\left(c_\e-2\log\lam_\e\right)\int_{\S\setminus{B_{Rr_\e}(x_\e)}} he^{u_\e} dv_g\nonumber\\
=&\left(c_\e-2\log\lam_\e\right) \int_{\p B_{Rr_\e}(x_\e)}\f{\p u_\e}{\p n}ds_g-\left(c_\e-2\log\lam_\e\right)\frac{8\pi(1-\e)}{\int_\S \psi dv_g}\int_{\S\setminus B_{Rr_\e}(x_\e)}\psi dv_g.
\end{align}
It follows from (\ref{out-is-0}) that
\begin{align}
\label{out3}
 8\pi(1-\e)\lam_\e^{-1}\int_{\S\setminus{B_{Rr_\e}(x_\e)}}\le(-2\log\pi-2\log h(p)-A_{p}+o_\epsilon(1)+o_R(1)\ri) he^{u_\e}dv_g=o_\e(1).
\end{align}
Inserting (\ref{out1})-(\ref{out3}) into (\ref{out0}) one obtains that
\begin{align}\label{out-first}
&8\pi(1-\e)\lam_\e^{-1}\int_{\S\setminus{B_{Rr_\e}(x_\e)}} u_\e he^{u_\e} dv_g\nonumber\\
\geq&\frac{8\pi}{\int_\S \psi dv_g}\int_{B_{Rr_\e}(x_\e)}u_\e\psi dv_g- \frac{8\pi(1-\e)}{\int_\S \psi dv_g} \int_{B_{Rr_\e}(x_\e)}G_{x_\e}\psi dv_g\nonumber\\
 &+\int_{\p B_{Rr_\e}(x_\e)}G_{x_\e}\f{\p u_\e}{\p n}ds_g-\int_{\p B_{Rr_\e}(x_\e)}u_\e \f{\p G_{x_\e}}{\p n}ds_g\nonumber\\
&+\left(c_\e-2\log\lam_\e\right) \int_{\p B_{Rr_\e}(x_\e)}\f{\p u_\e}{\p n}ds_g\nonumber\\
&-\left(c_\e-2\log\lam_\e\right)\frac{8\pi(1-\e)}{\int_\S \psi dv_g}\int_{\S\setminus B_{Rr_\e}(x_\e)}\psi dv_g+o_\e(1).
\end{align}
Substituting (\ref{out-first}) into (\ref{out}) we have
\begin{align}
\label{out20}
\int_{\S\setminus{B_{Rr_\e}(x_\e)}}|\na_g u_\e|^2dv_g
\geq&\frac{8\pi}{\int_\S \psi dv_g}\int_{B_{Rr_\e}(x_\e)}u_\e\psi dv_g- \frac{8\pi(1-\e)}{\int_\S \psi dv_g} \int_{B_{Rr_\e}(x_\e)}G_{x_\e}\psi dv_g\nonumber\\
 &+\int_{\p B_{Rr_\e}(x_\e)}G_{x_\e}\f{\p u_\e}{\p n}ds_g-\int_{\p B_{Rr_\e}(x_\e)}u_\e \f{\p G_{x_\e}}{\p n}ds_g\nonumber\\
&+\left(c_\e-2\log\lam_\e\right) \int_{\p B_{Rr_\e}(x_\e)}\f{\p u_\e}{\p n}ds_g\nonumber\\
&-\left(c_\e-2\log\lam_\e\right)\frac{8\pi(1-\e)}{\int_\S \psi dv_g}\int_{\S\setminus B_{Rr_\e}(x_\e)}\psi dv_g\nonumber\\
&+\frac{8\pi(1-\e)}{\int_\S \psi dv_g}\int_{ B_{Rr_\e}(x_\e) }\psi u_\e dv_g-\int_{\p B_{Rr_\e}(x_\e)}u_\e\f{\p u_\e}{\p n}ds_g+o_\e(1).
\end{align}
Lemma \ref{bubble} tells us that
\begin{align}\label{out20_1}
\int_{B_{Rr_\e}(x_\e)} u_\e \psi dv_g=o_\e(1).
\end{align}
From (\ref{green-function-expression}) one knows
\begin{align}
\label{out20_2}
\int_{B_{Rr_\e}(x_\e)}G_{x_\e}\psi dv_g=o_\e(1).
\end{align}
Using Lemmas \ref{bubble} and  \ref{m-p}, one has
\begin{align}
\label{out20_3}
&-\int_{\p B_{Rr_\e}(x_\e)}\f{\p u_\e}{\p n}\le(u_\e-G_{x_\e}+c_\e-2\log\lam_\e\ri)ds_g\nonumber\\
\geq& \f{8\pi^2h(p)R^2}{1+\pi h(p)R^2}\le(
-2\log\pi-2\log h(p)-A_{x_\epsilon}\ri)+o_\e(1)+o_R(1)\nonumber\\
=&-16\pi\log\pi-16\pi\log h(p)-8\pi A_p+o_\e(1)+o_R(1).
\end{align}
In view of (\ref{bubble-function}) and (\ref{green-function-expression}) we have
\begin{align}
\label{out20_4}
-\int_{\p B_{Rr_\e}(x_\e)}u_\e \f{\p G_{x_\e}}{\p n}ds_g
=&-\le(c_\e-2\log(1+\pi h(p)R^2)+o_\e(1)\ri)\le(-8\pi+O(Rr_\e)\ri)\nonumber\\
=&8\pi c_\e-16\pi\log(1+\pi h(p)R^2)+o_\e(1)+o_R(1).
\end{align}
It is clear that by Lemmas \ref{l>0} and \ref{lemma-c}
\begin{align}\label{out20_5}
-\left(c_\e-2\log\lam_\e\right)\frac{8\pi(1-\e)}{\int_\S \psi dv_g}\int_{\S\setminus B_{Rr_\e}(x_\e)}\psi dv_g
=-8\pi(1-\e)\left(c_\e-2\log\lam_\e\right)+o_\e(1).
\end{align}
Therefore, by inserting (\ref{out20_1})-(\ref{out20_5}) into (\ref{out20}) we obtain
\begin{align}\label{out-last}
\int_{\S\setminus B_{Rr_\e}(x_\e)}|\na_g u_\e|^2 dv_g
\geq& -16\pi\log\pi-16\pi\log h(p)-8\pi A_p\nonumber\\
&+8\pi\e c_\e-16\pi\log(1+\pi h(p)R^2)\nonumber\\
&+16\pi(1-\e)\log\lam_\e+o_\e(1)+o_R(1).
\end{align}
Substituting (\ref{energy-in}) and (\ref{out-last}) into (\ref{energy}) one has
\begin{align}\label{lower}
J^{\psi,h}_\e(u_\e)=&\f{1}{2}\int_\S |\na_g u_\e|^2dv_g - 8\pi(1-\e)\log\lam_\e\nonumber\\
\geq&-8\pi-8\pi\log\pi-8\pi\log h(p)-4\pi A_p\nonumber\\
&+4\pi\e c_\e+o_\e(1)+o_R(1)
\end{align}
Letting $\e\ra0$ first, and then $R\ra+\infty$ in (\ref{lower}), we have
\begin{align}\label{lower-bound}
\inf_{u\in W^{1,2}(\S, g)}J^{\psi,h}(u)=\lim_{\e\ra0}J^{\psi,h}_\e(u_\e)
&\geq-8\pi-8\pi\log\pi-8\pi\log h(p)-4\pi A_p\nonumber\\
&\geq-8\pi-8\pi\log\pi-4\pi\max_{y\in\S\setminus Z}\left(2\log h(y)+A_y\right).
\end{align}
Then Theorem \ref{thm-m-t} follows directly. $\hfill{\square}$

\section{Proof of Theorem \ref{existence}---Part I: $h>0$ }

In this section, we shall construct a blowup sequence $\{\phi_\e\}_{\e>0}$ with
$$J^{\psi,h}(\phi_\e)<-8\pi-8\pi\log\pi-4\pi\max_{y\in \S}\le(2\log h(y)+A_y\ri)$$
for sufficiently small $\e$. This is a contradiction with (\ref{lower-bound}), so no blowup happens, then we are in the position of Case a)
and the proof terminates. In fact, our proof also shows that, if $J^{\psi,h}$ has no minimizer in $W^{1,2}(\S,g)$, then
\begin{align*}\inf_{u\in W^{1,2}(\S,g)}J^{\psi,h}(u)=-8\pi-8\pi\log\pi-4\pi\max_{y\in \S}\le(2\log h(y)+A_y\ri).
\end{align*}

\underline{\textit{Proof of Theorem \ref{existence}---Part I: $h>0$.}} Suppose that $2\log h(p)+A_p =\max_{y\in \S}\le(2\log h(y)+A_y\ri)$. Let $r=\text{dist}(x,p)$.
Denote
\begin{align*}
\beta(r,\theta)=&G_p-\le(-4\log r+A_p+b_1r\cos\theta+b_2r\sin\theta\ri)\nonumber\\
=&c_1r^2\cos^2\theta+2c_2r^2\cos\theta\sin\theta+c_3r^2\sin^2\theta+O(r^3).
\end{align*}
We define
\begin{align}\label{test-function}
\phi_\e=
\begin{cases}
-2\log\le(r^2+\e\ri)+b_1r\cos\theta+b_2r\sin\theta+\log \e,&~~r\leq \a\sqrt{\e},\\
\le(G_p -\eta\beta(r,\theta)\ri)+C_\e+\log \e,&~~\a\sqrt{\e}\leq r \leq 2\a\sqrt{\e},\\
G_p+C_\e+\log \e, &~~r\geq 2\a\sqrt{\e},
\end{cases}
\end{align}
where $\eta\in C_0^{1}\le(B_{2\a\sqrt{\e}}(p)\ri)$ is a cutoff function satisfying
$\eta\equiv1$ in $B_{\a\sqrt{\e}}(p)$ and $|\na_g \eta|\leq\f{C}{\a\sqrt{\e}}$,
$C_\e=-2\log\f{\a^2+1}{\a^2}-A_p$, $\a=\a(\e)$ satisfying $\a\ra\infty$ and $\a\sqrt{\e}\ra0$
as $\e\ra0$ will be determined later.

We denote by $(r,\theta)$ the chosen normal coordinate system around $p$. We
write $g = dr^2 + g^2(r,\theta)d\theta^2$. It is well-known that
\begin{align}\label{metric}
g(r,\theta)=r-\f{K_g(p)}{6}r^3+O(r^4).
\end{align}
Using (\ref{metric}) and calculating directly, we have
\begin{align}\label{phi_energy}
\int_\S |\na_g \phi_\e|^2 dv_g
=&16\pi\log\f{\a^2+1}{\a^2}-16\pi\log\e-16\pi+16\pi\f{1}{1+\a^2}\nonumber\\
&+8\pi A_p+\f{16}{3}\pi K(p)\e\log(\a^2+1)+O\le(\a^4\e^2\log(\a^2\e)\ri)
\end{align}
and
\begin{align}\label{phi_h-int}
\int_\S he^{\phi_\e} dv_g
=& \pi h(p)\f{\a^2}{\a^2+1}\le[1+\f{1}{\a^2+1}-\f{1}{6}K(p)\f{\a^2+1}{\a^2}\e\log\le(\a^2+1\ri)\ri.\nonumber\\
&+\f{1}{4}\f{a^2+1}{\a^2}\le(b_1^2+b_2^2\ri)\e\log\le(\a^2+1\ri)-\f{1}{4}\f{\a^2}{\a^2+1}\le(b_1^2+b_2^2\ri)\e\log\le(\a^2\e\ri)\nonumber\\
&-\f{1}{2}\f{\a^2}{\a^2+1}(c_1+c_3-\f{1}{3}K(p))\e\log\le(\a^2\e\ri)\nonumber\\
&+\f{1}{4}\f{\a^2+1}{\a^2}\f{\la_g h(p)}{h(p)}\e\log\le(\a^2+1\ri)-\f{1}{4}\f{\a^2}{\a^2+1}\f{\la_g h(p)}{h(p)}\e\log\le(\a^2\e\ri)\nonumber\\
&+\le.\f{1}{2}\f{\a^2+1}{\a^2}\f{k_1b_1+k_2b_2}{h(p)}\e\log\le(\a^2+1\ri)-\f{1}{2}\f{\a^2}{\a^2+1}\f{k_1b_1+k_2b_2}{h(p)}\e\log\le(\a^2\e\ri)\ri]\nonumber\\
&+O\le(\a^4\e^2\ri)+O(\e).
\end{align}
We refer the readers to pages 241 and 245 in \cite{DJLW97} for the details of calculation of (\ref{phi_energy}) and (\ref{phi_h-int}).

Suppose that
\begin{align*}
\psi(x)-\psi(p)=&l_1r\cos\theta+l_2r\sin\theta+l_3r^2\cos^2\theta+2l_4r^2\sin\theta\cos\theta+l_5r^2\sin^2\theta+O(r^3)
\end{align*}
in $B_\d(p)$ for a small $\d>0$.

Direct computations tell us that
\begin{align*}
\int_{B_{\a\sqrt{\e}}(p)}\le(\psi-\psi(p)\ri)\le(-2\log(r^2+\e)+b_1r\cos\theta+b_2r\sin\theta\ri) dv_g=O\le(\a^4\e^2\log(\a^2\e\ri)
\end{align*}
and
\begin{align*}
&\int_{B_{\a\sqrt{\e}}(p)}\le(-2\log(r^2+\e)+b_1r\cos\theta+b_2r\sin\theta\ri) dv_g\nonumber\\
=&-2\pi\a^2\e\log\le((\a^2+1)\e\ri)-2\pi\e\log\le(\a^2+1\ri)+2\pi\a^2\e+O\le(\a^4\e^2\log(\a^2\e\ri).
\end{align*}
So
\begin{align}\label{psi-phi-in}
\int_{B_{\a\sqrt{\e}}(p)}\psi\phi_\e dv_g
=&\psi(p)\le[-2\pi\a^2\e\log\le((\a^2+1)\e\ri)-2\pi\e\log\le(\a^2+1\ri)+2\pi\a^2\e\ri]\nonumber\\
  &+\log\e\int_{B_{\a\sqrt{\e}}(p)}\psi dv_g+O\le(\a^4\e^2\log(\a^2\e\ri).
\end{align}
We have
\begin{align}\label{psi-phi-out-0}
&\int_{\S\setminus B_{\a\sqrt{\e}}(p)}\psi \le((G_p-\eta\beta(r,\th))+C_\e+\log\e\ri)dv_g\nonumber\\
=&\int_{\S\setminus B_{\a\sqrt{\e}}(p)}\psi G_pdv_g-\int_{B_{2\a\sqrt{\e}}(p)\setminus B_{\a\sqrt{\e}}(p)}\psi\eta\beta(r,\th)dv_g
  +\le(C_\e+\log\e\ri)\int_{\S\setminus B_{\a\sqrt{\e}}(p)}\psi dv_g.
\end{align}
By a direct calculation, one has
\begin{align}\label{out_I}
\int_{\S\setminus B_{\a\sqrt{\e}}(p)}\psi G_pdv_g
=&-\int_{B_{\a\sqrt{\e}}(p)}\psi G_p dv_g\nonumber\\
=&-\int_{B_{\a\sqrt{\e}}(p)}\le(\psi-\psi(p)\ri) G_p dv_g-\int_{B_{\a\sqrt{\e}}(p)}\psi(p) G_p dv_g\nonumber\\
=&-\psi(p)\le(-2\pi\a^2\e\log\le(\a^2\e\ri)+2\pi\a^2\e+\pi A_p \a^2\e\ri)+O\le(\a^4\e^2\log(\a^2\e\ri)
\end{align}
where we have used $$\int_{B_{\a\sqrt{\e}}(p)}\psi dv_g=\pi\psi(p)\a^2\e+O\le(\a^4\e^2\ri).$$
Meanwhile we have
\begin{align}\label{out_II}
-\int_{B_{2\a\sqrt{\e}}(p)\setminus B_{\a\sqrt{\e}}(p)}\psi\eta\beta(r,\th)dv_g=O\le(\a^4\e^2\ri).
\end{align}
Inserting (\ref{out_I}) and (\ref{out_II}) into (\ref{psi-phi-out-0}), we have
\begin{align}\label{psi-phi-out}
\int_{\S\setminus B_{\a\sqrt{\e}}(p)}\psi\phi_\e dv_g
=&-\psi(p)\le(-2\pi\a^2\e\log\le(\a^2\e\ri)+2\pi\a^2\e+\pi A_p \a^2\e\ri)\nonumber\\
  &+(C_\e+\log\e)\int_{\S\setminus B_{\a\sqrt{\e}}(p)}\psi dv_g+O\le(\a^4\e^2\log(\a^2\e\ri).
\end{align}
Combining (\ref{psi-phi-in}) and (\ref{psi-phi-out}), we have
\begin{align}\label{psi_phi}
\int_\S \psi \phi_\e dv_g
=&\int_{B_{\a\sqrt{\e}}(p)}\psi\phi_\e dv_g+\int_{\S\setminus B_{\a\sqrt{\e}}(p)}\psi\phi_\e dv_g\nonumber\\
=&\int_\S \psi dv_g\le(\log\e-2\log\f{\a^2+1}{\a^2}-A(p)-2\pi\frac{\psi(p)}{\int_\S \psi dv_g}\e\log(\a^2+1)\ri)\nonumber\\
  & + O\le(\a^4\e^2\log(\a^2\e\ri).
\end{align}
Then by (\ref{phi_energy}), (\ref{phi_h-int}) and (\ref{psi_phi}) one has
\begin{align*}
J^{\psi,h}(\phi_\e)
=&\f{1}{2}\int_\S |\na_g \phi_\e|^2 dv_g+8\pi\f{1}{\int_\S \psi dv_g}\int_\S \psi\phi_\e dv_g-8\pi\log\int_\S he^{\phi_\e} dv_g\nonumber\\
=&-8\pi-8\pi\log\pi-4\pi A_p-8\pi\log h(p)\nonumber\\
  &-16\pi^2\le(\f{\psi(p)}{\int_\S \psi dv_g}-\f{1}{4\pi}K_g(p)\ri)\e\log\le(\a^2+1\ri)+16\pi^2\le(1-\f{1}{4\pi}K_g(p)\ri)\e\log\le(\a^2\e\ri)\nonumber\\
  &-2\pi\f{\a^2+1}{\a^2}\le(b_1^2+b_2^2\ri)\e\log\le(\a^2+1\ri)+2\pi\f{\a^2}{\a^2+1}\le(b_1^2+b_2^2\ri)\e\log\le(\a^2\e\ri)\nonumber\\
&-2\pi\f{\a^2+1}{\a^2}\f{\la_g h(p)}{h(p)}\e\log\le(\a^2+1\ri)+2\pi\f{\a^2}{\a^2+1}\f{\la_g h(p)}{h(p)}\e\log\le(\a^2\e\ri)\nonumber\\
&-4\pi\f{\a^2+1}{\a^2}\f{k_1b_1+k_2b_2}{h(p)}\e\log\le(\a^2+1\ri)+4\pi\f{\a^2}{\a^2+1}\f{k_1b_1+k_2b_2}{h(p)}\e\log\le(\a^2\e\ri)\nonumber\\
&+O\le(\f{\e\log(\a^2+1)}{\a^2}\ri)+O\le(\f{-\e\log(\a^2\e)}{\a^2}\ri)+O\le(\le(\e\log(\a^2+1)\ri)^2\ri)\nonumber\\
&+O\le(\le(-\e\log(\a^2\e)\ri)^2\ri)+O\le(\f{1}{\a^4}\ri)+O\le(\a^4\e^2\ri)+O(\e),
\end{align*}
where we have used Proposition 3.2 in \cite{DJLW97}, i.e., $c_1+c_3+\f{2}{3}K_g(p)=4\pi$.
By choosing $\a=\le(\e\log(-\log\e)\ri)^{-1/4}$, we have
\begin{align}\label{bound-from-above}
J^{\psi,h}(\phi_\e)=&-8\pi-8\pi\log\pi-4\pi A_p-8\pi\log h(p)\nonumber\\
                 &-16\pi^2\le(\f{1}{2}\le(\f{\psi(p)}{\int_\S \psi dv_g}+1\ri)-\f{1}{4\pi}K_g(p)+\f{b_1^2+b_2^2}{8\pi}+\f{\la_g h(p)}{8\pi h(p)}+\f{k_1b_1+k_2b_2}{4\pi h(p)}\ri)\e(-\log\e)\nonumber\\
                 &+o(\e(-\log\e)).
\end{align}
So if (\ref{condition}) is satisfied, by (\ref{bound-from-above}) we have
\begin{align*}
J^{\psi,h}(\phi_\epsilon)<-8\pi-8\pi\log\pi-4\pi\max_{y\in \Sigma}\left(2\log h(y)+A_y\right)
\end{align*}
for sufficiently small $\epsilon>0$.
This ends the proof of Theorem \ref{existence} when $h>0$.   $\hfill{\square}$

\section{Proof of Theorem \ref{existence}---Part II: $h\geq0$, $h\nequiv0$ }

In this section, we shall deal with the situation that $h\geq0$ and $h\nequiv0$ on $\S$ and end the proof of Theorem \ref{existence}. The idea comes from our paper \cite{YZ2016}.\\

First, we need the following concentration lemma, which can be seen as a generalization of one on $S^2$ proved by
Chang-Yang \cite{CY88} and one on a general compact Riemannian surface proved by Ding-Jost-Li-Wang \cite{DJLW98}.
\begin{prop}\label{concentration-lemma}
Let $(\S, g)$ be a compact Riemannian surface. Let $\psi$ be a smooth function on $\Sigma$ satisfying $\int_\Sigma \psi dv_g\neq0$.
Given a sequence of $u_j\in W^{1,2}(\S, g)$ with $\int_\S e^{u_j}dv_g=1$ and
$$\f{1}{2}\int_\S |\na_g u_j|^2 dv_g+8\pi\f{1}{\int_\S \psi dv_g}\int_\S \psi u_j dv_g\leq C.$$
Then either

$(i)$ there exists a constant $C_0>0$ such that $\int_\S |\na_g u_j|^2 dv_g\leq C_0$ or

$(ii)$ there exists a subsequence which is also denoted by $u_j$ concentrates at a point $p\in \S$, i.e., for any $r>0$,
$$\lim_{j\ra\infty}\int_{B_r(p)}e^{u_j}dv_g=1.$$
\end{prop}
To prove Proposition \ref{concentration-lemma}, one needs the following "distribution of mass" lemma, which can be seen as a generalization of
one proved by Aubin \cite{Au82} (see also \cite{CL91}).
\begin{lemma}\label{distribution-of-mass}
Let $(\S, g)$ be a compact Riemannian surface. Let $\psi$ be a smooth function on $\Sigma$ satisfying $\int_\Sigma \psi dv_g\neq0$.
Let $\O_1$ and $\O_2$ be two subsets of $\S$ satisfying $dist(\O_1, \O_2)\geq\e_0>0$ and $\a_0\in (0, 1/2)$. For any $\e\in(0,1)$,
there exists a constant $C=C(\e,\e_0,\a_0)$ such that
\begin{align*}
\int_\S e^u dv_g\leq C\exp\le\{\f{1}{32\pi(1-\e)}||\na_g u||_2^2+\wan{u}\ri\}
\end{align*}
holds for any $u\in W^{1,2}(\S, g)$ satisfying
\begin{align}\label{distribution}
\f{\int_{\O_1}e^udv_g}{\int_\S e^u dv_g}\geq\a_0~~~~\text{and}~~~~\f{\int_{\O_2}e^udv_g}{\int_\S e^u dv_g}\geq\a_0.
\end{align}
\end{lemma}
\begin{proof}
Let $\phi_1$, $\phi_2$ be two smooth functions on $\S$ such that
$$0\leq\phi_i\leq1,~~\phi_i\equiv1,~~\text{for}~x\in\O_i,~~i=1,2$$
and $\text{supp}~\phi_1 \cap \text{supp}~\phi_2=\emptyset$. It suffices to show that for $u\in W^{1,2}(\S, g)$, $\wan{u}=0$,
(\ref{distribution}) implies
\begin{align}\label{improved-m-t}
\int_\S e^u dv_g\leq C \exp\le\{\f{1}{32\pi(1-\e)}||\na_g u||^2_2\ri\}.
\end{align}

We can assume without loss of generality that $||\na_g(\phi_1 u)||_2\leq ||\na_g (\phi_2 u)||_2$. Then by (\ref{distribution}) and
 Theorem \ref{thm-m-t},  one has
\begin{align}\label{distribution1}
\int_\S e^u dv_g
\leq&  \f{1}{\a_0}\int_{\O_1} e^u dv_g\leq \f{1}{\a_0}\int_\S e^{\phi_1 u} dv_g\nonumber\\
\leq& \f{C}{\a_0} \exp\le\{\f{1}{16\pi}||\na_g (\phi_1 u)||^2_2+\wan{\phi_1 u}\ri\}\nonumber\\
\leq& \f{C}{\a_0} \exp\le\{\f{1}{32\pi}||\na_g \le[(\phi_1+\phi_2)u\ri]||^2_2+\wan{\phi_1 u}\ri\}\nonumber\\
\leq& \f{C(\e_0)}{\a_0} \exp\le\{\f{1}{32\pi}(1+\e_1)||\na_g u||^2_2+C(\e_1)||u||_2^2\ri\}
\end{align}
for some small $\e_1>0$, where in the last inequality we have used Lemma \ref{lemma-poincare-ineq} and the Cauchy's inequality.

Using the condition $\wan{u}=0$ one can get rid of the term $||u||_2^2$ on the right hand side of (\ref{distribution1}).

Given small enough $\eta>0$, there exists $a_{\eta}$ such that $\text{Vol}_g\{x\in \S:~u(x)\geq a_{\eta}\}=\eta$. Applying (\ref{distribution1}) to the function
$(u-a_{\eta})_+= \max\{0, (u-a_{\eta})\}$, we have
\begin{align}\label{distribution2}
\int_\S e^u dv_g\leq& e^{a_{\eta}}\int_\S e^{(u-a_{\eta})} dv_g \leq e^{a_{\eta}} \int_\S e^{(u-a_{\eta})_+} dv_g\nonumber\\
\leq& C \exp\le\{\f{1}{32\pi}(1+\e_1)||\na_g u||_2^2 +C(\e_1)||(u-a_{\eta})_+||_2^2 + a_{\eta}\ri\},
\end{align}
where $C=C(\e_0, \a_0)$.

By the H\"{o}lder inequality and Lemma \ref{lemma-s-p}, we have
\begin{align}\label{distribution3}
\int_\S \le |(u-a_{\eta})_+\ri|^2 dv_g=&\int_{\{x\in\S:~u(x)\geq a_{\eta}\}} \le |(u-a_{\eta})_+\ri|^2 dv_g\nonumber\\
\leq& \le(\int_{\{x\in\S:~u(x)\geq a_{\eta}\}} \le |(u-a_{\eta})_+\ri |^4\ri)^{1/2}\cdot\eta^{1/2}\nonumber\\
\leq& \le(\int_\S |u|^4 dv_g\ri)^{1/2}\cdot\eta^{1/2}\nonumber\\
\leq& C \int_\S |\na_g u|^2 dv_g \cdot \eta^{1/2}.
\end{align}
By the H\"{o}der inequality and Lemma \ref{lemma-poincare-ineq}, we have
\begin{align*}
a_{\eta}\cdot\eta\leq\int_{\{x\in\S:~u(x)\geq a_{\eta}\}} u dv_g\leq \int_\S |u| dv_g\leq C\le(\int_\S |\na_g u|^2 dv_g\ri)^{1/2}.
\end{align*}
So
\begin{align}\label{distribution5}
a_{\eta}\leq \eta \int_\S |\na_g u|^2 dv_g+\f{C}{\eta}.
\end{align}
Substituting (\ref{distribution3}) and (\ref{distribution5}) into (\ref{distribution2}) and choosing $\e_1$ and $\eta$ sufficiently small such that
\begin{align*}
\f{1}{32\pi}(1+\e_1)+C(\e_1)C\eta^{1/2}+\eta\leq\f{1}{32\pi(1-\e)},
\end{align*}
then we obtain the inequaliaty (\ref{improved-m-t}). This ends the proof of Lemma \ref{distribution-of-mass}.
\end{proof}

\underline{\textit{Proof of Proposition \ref{concentration-lemma}.}} If $(ii)$ does not hold, i.e., every subsequence of $u_j$ does not concentrate. Then
for any $p\in \S$ there exists $r\in(0,i_\S/16)$ such that
\begin{align}\label{prop1}
\lim_{j\ra\infty} \int_{B_r(p)} e^{u_j} dv_g < \d_0 <1,
\end{align}
where $i_\S$ is the injective radius of $(\S, g)$. (Note that we do not distinguish sequence and its subsequences.)

Since $(\S, g)$ is compact, there exists a finite set $\{(p_l, r_l):~l=1,2,...,L\}$ satisfying
\begin{align*}
\lim_{j\ra\infty}\int_{B_{r_l}(p_l)} e^{u_j} dv_g<\d<1
\end{align*}
and $\bigcup_{l=1}^{L} B_{r_l}(p_l)=\S$.

Without loss of generality, one may assume that
\begin{align*}
\lim_{j\ra\infty} \int_{B_{r_1}(p_1)} e^{u_j} dv_g \geq \a_0>0
\end{align*}
where $\a_0\in(0,\d_0)$ is a constant.

We prove $(i)$ must happen by contradiction. Suppose $(i)$ not happens, i.e., $\int_\S |\na_g u_j|^2 dv_g$ is unbounded.
Then from Lemma \ref{distribution-of-mass} we know
\begin{align}\label{distribution7}
\lim_{j\ra\infty} \int_{\S\setminus B_{2r_1}(p_1)} e^{u_j} dv_g=0.
\end{align}
Choosing a normal coordinate system $(x_1, x_2)$ around $p_1$ and assuming in $B_{16r_1}(p_1)$
\begin{align*}
\f{1}{2}|x-y| \leq dist_g(x,y) \leq 2|x-y|,
\end{align*}
where $|x-y|=dist_{\mathbb{R}^2}(x,y)$.

We consider the square $P_1=\{ |x_i|\leq 4r_1:~i=1,2\}\subset \mathbb{R}^2$, from (\ref{distribution7}) one knows
\begin{align*}
\lim_{j\ra\infty}\int_{\exp_{p_1}(P_1)} e^{u_j} dv_g = 1.
\end{align*}
Dividing $P_1$ into $16$ equal sub-squares. Since $\int_\S |\na_g u_j|^2 dv_g$ is unbounded, by Lemma \ref{distribution-of-mass}
one gets a square $P_2$ which is a union of at most $9$ of the equal sub-squares of $P_1$ such that
\begin{align*}
\lim_{j\ra\infty}\int_{\exp_{p_1}(P_2)} e^{u_j} dv_g = 1.
\end{align*}
Continuing this procedure, we can obtain a sequence of square $P_n$. It is easy to check that $P_n\ra p_0$ as $n\ra\infty$ for some $p_0\in\S$
and
\begin{align*}
\lim_{j\ra\infty}\int_{B_r(p_0)} e^{u_j} dv_g = 1
\end{align*}
for any $r\in(0, i_\S)$. This contradicts (\ref{prop1}). The contradiction tells us that $\int_\S |\na_g u_j|^2 dv_g$ is bounded,
i.e., $(i)$ holds. This ends the proof of Proposition \ref{concentration-lemma}.   $\hfill{\square}$\\

Now, we are ready to prove Theorem \ref{existence} when $h\geq0$, $h\nequiv0$. \\

\underline{\textit{ Proof of Theorem \ref{existence}---Part II: $h\geq0$, $h\nequiv0$.}}
Checking the proof of Theorem \ref{existence}---Part I: $h>0$ carefully, one will finds that
the condition $h>0$ is just used in solving the bubble (\ref{bubble-function}). Therefore, if $h\geq0$, $h\nequiv0$,
we just need to prove that the blowup (if happens) will not happen on zero point of $h$.

In the following we assume $u_\e$ blows up, i.e., $||\na u_\e||_2\ra+\infty$ as $\e\ra0$.

Recalling that in (\ref{E-L_equation}), we have $u_\epsilon\in C^{\infty}(\S)\cap \wan{X}$.
Lemma \ref{lemma-c} still holds, i.e., $c_\e\ra\infty$ as $\e\ra0$.
Let $\O\subset \S$ be a domain. If
$\lam_\e^{-1}\int_\O he^{u_\e}dv_g\leqslant\frac{1}{2}-\d$
for some $\d\in(0, \f{1}{2})$, then (\ref{out-infty-norm}) implies that
\be\label{Fact3}||u_\e||_{L^{\infty}(\O_0)}\leq C(\O_0,\O),\quad\forall \O_0\subset\subset\O.\ee

Suppose $x_\e\ra p$ as $\e\ra0$. As we explained before, to prove Theorem \ref{existence}, it
 suffices to prove that $h(p)>0$. For this purpose, we set
$v_\e=u_\e-\log\int_\S e^{u_\e}dv_g$. Then we have
\begin{align}\label{h=0_1}
\int_\S e^{v_\e}dv_g=1,~~J^{\psi,h}_\e(v_\e)=J^{\psi,h}_\e(u_\e)
\end{align}
 and
\begin{align}\label{h=0_2}
J^{\psi,h}_\e(u_\e)=\inf_{u\in \wan{X}}J^{\psi,h}_\e(u)\leq J^{\psi,h}_\e(0)=-8\pi(1-\e)\log\int_\S h dv_g.
\end{align}
The H\"{o}lder inequality together with Lemmas \ref{lemma-poincare-ineq} and \ref{lemma-l_q-norm} tells us that
\begin{align}\label{mean-u}
|\bar{u}_\e|=|\wan{u}_\e-\bar{u}_\e|
=\le|\f{1}{\int_\S \psi dv_g}\int_\S \psi\le(u_\e-\bar{u}_\e\ri) dv_g\ri|
\leq C
\end{align}
By Jensen's inequality and (\ref{mean-u}), we have
\begin{align}\label{mean-of-v}
\wan{v}_\e=-\log\int_\S e^{u_\e} dv_g&=-\log\mbox{Vol}_g(\S)-\log\left(\f{1}{\mbox{Vol}_g(\S)}\int_\S e^{u_\e} dv_g\right)\nonumber\\
&\leq-\log\mbox{Vol}_g(\S)+C\nonumber\\
&\leq C.
\end{align}
Combining (\ref{h=0_1}), (\ref{h=0_2}) and (\ref{mean-of-v}) one obtains
\begin{align}
\f{1}{2}\int_\S |\na_g v_\e|^2dv_g+8\pi\wan{v}_\e
\leq& J^{\psi,h}_\e(v_\e)+8\pi\e\wan{v}_\e+8\pi(1-\e)\log\int_\S he^{v_\e}dv_g \nonumber\\
\leq& J^{\psi,h}_\e(u_\e)+8\pi\e\wan{v}_\e+8\pi(1-\e)\log\max_\S h \nonumber\\
\leq& C + 8\pi(1-\e)\log\frac{\max_\S {h}}{\int_\S h dv_g}\nonumber\\
\leq&C.\nonumber
\end{align}
Clearly, $(ii)$ of Lemma \ref{concentration-lemma} holds in this case. Hence there exists some $p^\prime\in \S$ such that
$v_\e$ concentrates at $p'$, namely,
\begin{align}\label{conc}
\lim_{\e\ra0}\int_{B_r(p')}e^{v_\e}dv_g=1,\quad\forall r>0.
\end{align}
We first claim that
\be\label{claim}
h(p')>0.\ee
To see this, in view of (\ref{conc}), we calculate
\begin{align}
\label{cal}
\frac{ \int_\S he^{u_\e}dv_g}{\int_\S e^{u_\e}dv_g}=\int_\S he^{v_\e}dv_g&=\int_{B_r(p')}he^{v_\e}dv_g+
\int_{\S\setminus B_r(p')}he^{v_\e}dv_g\nonumber\\
                          &=(h(p')+o_r(1))\int_{B_r(p')}e^{v_\e}dv_g+o_\e(1)\nonumber\\
                          &=(h(p')+o_r(1))(1+o_\e(1))+o_\e(1)\nonumber\\
                          &=h(p')+o_r(1)+o_\e(1).
\end{align}
Because the left hand side of (\ref{cal}) does not depend on $r$,  we have by passing to the limit
$r\ra 0$,
\begin{align}
\label{pp}
\frac{ \int_\S he^{u_\e}dv_g}{\int_\S e^{u_\e}dv_g}=h(p')+o_\e(1).
\end{align}
So by (\ref{pp}) and Theorem \ref{thm-m-t} we have
\begin{align}
\label{18}
J^{\psi,h}_\e(u_\e)=&\frac{1}{2}\int_\S |\na_g u_\e|^2dv_g-8\pi(1-\e)\log\int_\S he^{u_\e}dv_g\nonumber\\
                =&\frac{1}{2}\int_\S |\na_g u_\e|^2dv_g-8\pi(1-\e)\log\le(h(p')+o_\e(1)\ri)\nonumber\\
                    &-8\pi(1-\e)\log\int_\S e^{u_\e}dv_g\nonumber\\
                 \geq&\frac{1}{2}\int_\S |\na_g u_\e|^2dv_g-8\pi(1-\e)\log\le(h(p')+o_\e(1)\ri)\nonumber\\
                    &-8\pi(1-\e)\log\le(C \exp\le\{\frac{1}{16\pi}\int_\S |\na_g u_\e|^2dv_g\ri\}\ri)\nonumber\\
                 \geq&-8\pi(1-\e)\log\le(C(h(p')+o_\e(1))\ri).
\end{align}
Combining (\ref{h=0_2}) and (\ref{18}), we obtain
\bna
-8\pi(1-\e)\log\int_\S h dv_g\geq-8\pi(1-\e)\log\le(C(h(p')+o_\e(1))\ri),
\ena
and whence
$$\log\int_\S h dv_g\leq\log\le(Ch(p')\ri).$$
This immediately leads to (\ref{claim}).

Then we claim that
\be\label{point}p^\prime=p.\ee
In view of $(ii)$ of Lemma \ref{concentration-lemma}, there holds $\int_\O e^{v_\e}dv_g\ra 0$,
$\forall \O\subset\subset \S\setminus\{p'\}$. Noting that $u_\epsilon\in \wan{X}$, from (\ref{pp}) we have
\begin{align}
\label{20}
\lam_\e^{-1}\int_\O he^{u_\e}dv_g=\lam_\e^{-1}\int_\S e^{u_\e}dv_g\int_\O he^{v_\e}dv_g\leq \frac{\max_{\S}h}{h(p')+o_\e(1)}
\int_\O e^{v_\e}dv_g\ra 0
\end{align}
as $\epsilon\ra 0$.
Combining (\ref{Fact3}) and (\ref{20}),  we obtain $||u_\e||_{L^{\infty}(\O)}\leq C$
and thus $||v_\e-\wan{v}_\e||_{L^{\infty}(\O)}\leq C$
for any $\O\subset\subset \S\setminus\{p'\}$. This together with (\ref{mean-of-v}) implies that $v_\e(x)\leq C$
for all $x\in \O\subset\subset \S\setminus\{p'\}$.

It follows from (\ref{pp}) and $u_\epsilon\in \wan{X}$ that
$$\lam_\e^{-1}\int_\S e^{u_\e}dv_g=\frac{1}{h(p')+o_\e(1)}<\frac{2}{h(p')}$$
for sufficiently small $\e>0$.
Suppose $p'\neq p$. Recalling that $c_\epsilon=u_\epsilon(x_\epsilon)=\max_{\S}u_\epsilon$ and
 $x_\epsilon\ra p$, we find a domain $\O$ such that $x_\e\in\O \subset\subset \S\setminus\{p'\}$. Hence
\bna
c_\e-\log\lam_\e=u_\e(x_\e)-\log\lam_\e= v_\e(x_\e)+\log\le(\lam_\e^{-1}\int_\S e^{u_\e}dv_g\ri)\leq C,
\ena
which contradicts Lemma \ref{lemma-c} and concludes our claim (\ref{point}). Combining (\ref{claim})
and (\ref{point}), we obtain $h(p)>0$. The remaining part of the proof of Theorem \ref{existence} is
completely analogous to Section 4, we omit the details here.     $\hfill\Box$\\

  {\bf Acknowledgements}. The work is supported by the National Science Foundation of China
  (Grant No. 11401575). The author thanks Professor Yunyan Yang for
  his helpful discussions and suggestions.

\end{document}